\documentclass[12pt]{amsart}
\usepackage{amsthm, amsmath, amssymb, url, verbatim, amscd, diagrams,enumitem} 
\usepackage[margin=1.25in]{geometry}
{\providecommand{\noopsort}[1]{} 
\numberwithin{table}{section}
\frenchspacing

\theoremstyle{plain}
\newtheorem{theorem}{Theorem}[section]
\newtheorem{corollary}[theorem]{Corollary}
\newtheorem{lemma}[theorem]{Lemma}
\newtheorem{proposition}[theorem]{Proposition}

\newtheorem{main}{Theorem}
\newtheorem{maincorollary}[main]{Corollary}
\theoremstyle{definition}
\newtheorem{definition}[theorem]{Definition}

\theoremstyle{remark}
\newtheorem{remark}[theorem]{Remark}

\numberwithin{equation}{section}

\newcommand{\Z}{\mathbb{Z}}\newcommand{\Q}{\mathbb{Q}}
\newcommand{\HH}{\mathbb{H}}

\newcommand{\Sp}{\mathrm{Sp}}	\renewcommand{\Sp}{\ensuremath{\operatorname{\mathsf{Sp}}}}
\newcommand{\Spin}{\mathrm{Spin}}\renewcommand{\Spin}{\ensuremath{\operatorname{\mathsf{Spin}}}}
	
\newcommand{\SO}{\mathrm{\SO}}	\renewcommand{\SO}{\ensuremath{\operatorname{\mathsf{SO}}}}
	
\newcommand{\SU}{\mathrm{SU}}	\renewcommand{\SU}{\ensuremath{\operatorname{\mathsf{SU}}}}
\newcommand{\Un}{\mathsf{U}}	\renewcommand{\Un}{\ensuremath{\operatorname{\mathsf{U}}}}

\newcommand{\sone}{\mathrm{S}^1}	\renewcommand{\sone}{\ensuremath{\operatorname{\mathsf{S^{\hspace{.005in}1}}}}}

\newcommand{\E}{\ensuremath{\operatorname{\mathsf{E}}}}
\newcommand{\F}{\ensuremath{\operatorname{\mathsf{F}}}}
\newcommand{\Gtwo}{\ensuremath{\operatorname{\mathsf{G_{\hspace{.005in}2}}}}}

\newcommand{\gT}{\ensuremath{\operatorname{\mathsf{T}}}}

\newcommand{\s}{\mathbb{S}}

\newcommand{\CP}{\mathbb{CP}}
\newcommand{\HP}{\mathbb{H\mkern1mu P}}
\newcommand{\OP}{\mathbb{O\mkern1mu P}}
\newcommand{\QP}{\mathbb{Q\mkern1mu P}}

\DeclareMathOperator{\cod}{cod}

\DeclareMathOperator{\Isom}{Isom}
\DeclareMathOperator{\rank}{rank}

\DeclareMathOperator{\rk}{rank}

\newcommand{\of}[1]{\left(#1\right)}

\newcommand{\tensor}{\otimes}

\title[Cohomogeneity one and singly generated rational cohomology]{Cohomogeneity one manifolds with singly generated rational cohomology}
\author{Jason DeVito}
\address{Department of Mathematics, University of Oklahoma, Norman, OK 73019}
\email{Jason.B.DeVito-1@ou.edu}

\author{Lee Kennard}
\address{Department of Mathematics, Syracuse University, Syracuse, NY 13244}
\email{ltkennar@syr.edu}

\date{\today}

\setcounter{tocdepth}{1}

\begin{document}


\begin{abstract} We classify simply connected, closed cohomogeneity one manifolds with singly generated or $4$-periodic rational cohomology and positive Euler characteristic.\end{abstract}

\maketitle

We denote by $\QP^n_k$ any smooth, simply connected, closed manifold whose rational cohomology is isomorphic to the truncated polynomial algebra $\Q[x]/(x^{n+1})$ where the generator $x$ has degree $k$. Note that such a manifold is a rational sphere or point if $k$ is odd by the graded commutativity of the cup product. If $k$ is even, then a $\QP^n_k$ has even dimension $kn$ and positive Euler characteristic $n+1$.

Prototypical examples are simply connected, closed manifolds with the rational cohomology (equivalently, rational homotopy type) of a compact rank one symmetric space: a rational sphere is a $\QP^1_k$, a rational $\CP^n$ is a $\QP^n_2$, a rational  $\HP^n$ is a $\QP^n_4$, and a rational Cayley plane is a $\QP^2_8$.

We call the parameters $n$ and $k$ standard if they correspond to a rank one symmetric space. The classification of parameters $(n,k)$ for which a $\QP^n_k$ exists is reduced by way of the Barge-Sullivan rational realization theorem to a number theoretic problem (see Su \cite{Su14}), but this problem is hard and there is no classification (cf. \cite{FowlerSu16,KennardSu}).

This paper is motivated in part by the attempt to find highly symmetric models for $\QP^n_k$ with non-standard parameters. It is known that no such manifold admits a homogeneous or biquotient structure (see Kapovitch--Ziller \cite{KapovitchZiller04} and Totaro \cite{Totaro02}). Our first theorem is a similar, negative result:

\begin{main}\label{thm:QP}
If a $\QP^n_k$ admits a cohomogeneity one action, then $n$ and $k$ are standard. If moreover  $k$ is even, then the space is diffeomorphic to a rank one symmetric space, the Grassmannian $\SO(2m+1)/\SO(2) \times \SO(2m-1)$, or $\Gtwo/\SO(4)$, and if the action is almost effective and has no orbit equivalent proper subaction, then it is equivalent to a linear action or, in the case of $\Gtwo/\SO(4)$, left multiplication by $\SU(3)\subseteq \Gtwo$. 
\end{main}

The rigidity statement for $n = 1$ answers a question posed by Amann and the second author about whether a simply connected cohomogeneity one manifold with Euler characteristic two must be a sphere (see \cite[Corollary C]{AmannKennard17}). Previously Straume classified cohomogeneity one actions on homotopy spheres and Asoh on mod two homology spheres (see \cite{Asoh81,Asoh83,Straume96}). Theorem \ref{thm:QP} extends these rigidity results in the even-dimensional case. As for the other cases with standard parameters, the rigidity statement is not new (see Uchida \cite{Uchida77} and Iwata \cite{Iwata78,Iwata81}). 

Our approach also allows us to nearly classify the diffeomorphism type of cohomogeneity one manifolds with even dimension and four-periodic rational cohomology (see Section \ref{sec:applications} for a definition). This cohomological condition arises, for example, by way of Wilking's connectedness lemma in the context of the Grove symmetry program, in which Riemannian manifolds with positive or non-negative sectional curvature and large symmetry are examined (see \cite{Grove17, Ziller14} for surveys). In this context, it is also natural to study homogeneous spaces, cohomogeneity one manifolds, and quotients thereof, as they have provided numerous examples of manifolds admitting positive or non-negative curvature (see \cite{Ziller07} for a survey and \cite{Dearricott11,GroveVerdianiZiller11,GoetteKerinShankar-pre} for more examples). Homogeneous spaces and, more generally, biquotients with four-periodic rational cohomology were classified by the first author (see \cite{DeVito18}). Our second main result is a step toward an analogous classification for cohomogeneity one manifolds.
\begin{main}\label{thm:4periodic}
A simply connected, closed manifold with the rational cohomology of $\s^2 \times \HP^n$ admits a cohomogeneity one action if and only if it is diffeomorphic to  $\s^2\times \HP^n$, $\s^2\times (\Gtwo/\SO(4))$, or the unique linear non-trivial $\HP^n$ bundle over $\s^2$.
\end{main}

A cohomogeneity one structure implies rational ellipticity (see Grove-Halperin \cite{GroveHalperin87}), and it is straightforward to show that an even-dimensional, simply connected, closed manifold with four-periodic rational cohomology and a cohomogeneity one action is a $\QP^1_k$, $\QP^n_2$, or $\QP^n_4$ or a rational $\s^2 \times \HP^n$ or a rational $\s^3 \times \s^3$ (see Proposition \ref{pro:model4periodic}). Hence Theorems \ref{thm:QP} and \ref{thm:4periodic} imply the following:
 
\begin{maincorollary}\label{cor:4periodic}
A simply connected, closed manifold with four-periodic rational cohomology and positive Euler characteristic admits a cohomogeneity one action if and only if it is diffeomorphic to $\s^n$, $\CP^n$, $\HP^n$, 
	 $\SO(2n+1)/\SO(2) \times \SO(2n-1)$, $\Gtwo/\SO(4)$,  $\s^2 \times \HP^n$, $\s^2 \times (\Gtwo/\SO(4))$,  or the non-trivial linear $\HP^n$-bundle over $\s^2$.
\end{maincorollary}

Note that, whereas the first seven examples in Corollary \ref{cor:4periodic} are symmetric spaces, the last one is not even homogeneous. On the other hand, it does arise as a biquotient. In fact, the lists of examples in Theorem \ref{thm:QP} and Corollary \ref{cor:4periodic} form proper subsets of the corresponding lists in the biquotient case (see \cite{KapovitchZiller04,KapovitchZiller04-arxiv} and \cite{DeVito18}). 

Finally we remark on our methods. While the use of rational homotopy theory was prevalent in the biquotient classifications cited above, it was not used in the cohomogeneity one results, nor was it used in the important work of Frank \cite{Frank13} that we apply here. Grove and Halperin \cite{GroveHalperin87} developed the Sullivan model for a cohomogeneity one manifold in terms of the homotopy fiber $\mathcal F$ of the inclusion $G/H \to M$, where $G/H$ is a principal orbit. One important consequence of their work is the rational ellipticity of cohomogeneity one manifolds. One of our new tools is a computation of the connecting homomorphisms of the associated fibration, which hold generally for cohomogeneity one manifolds with positive Euler characteristic (cf. \cite{DeVitoGalazGarciaKerin-pre}). Since $\mathcal F$ is computed in \cite{GroveHalperin87} and since we are assuming we know the rational homotopy of $M$, these computations imply strong restrictions on the groups $G$ and $H$. Combined with other  facts about cohomogeneity one manifolds, we are able to generalize and significantly shorten the proofs of some results in Uchida and Iwata's work, especially in cases where one or both of the singular orbits is non-orientable (see Propositions \ref{pro:bothsingnonorientable} and \ref{pro:onesingnonorientable}). 

After covering some preliminaries on cohomogeneity one manifolds in Section \ref{sec:preliminaries}, we describe in Section \ref{sec:RHT} the Grove-Halperin model and carry out the rational homotopy computations described in the previous paragraph. In Section \ref{sec:non-primitive}, we prove Theorems \ref{thm:QP} and \ref{thm:4periodic} in the non-primitive case, in which the manifold naturally fibers over a homogeneous space with a cohomogeneity one fiber. In Section \ref{sec:primitive}, we describes Frank's work on primitive cohomogeneity one actions on manifolds with positive Euler characteristic and elaborate on his work in preparation for the proofs of the main theorems. We then prove Theorems \ref{thm:QP} and \ref{thm:4periodic} in Sections \ref{sec:ProofQP} and \ref{sec:Proof4periodic}, respectively.

\subsection*{Acknowledgements}  Both authors would like to thank Vitali Kapovitch and Martin Kerin for helpful discussions at the early stages of this project. These conversations happened at MSRI and were in part supported by the institute through NSF Grant DMS-1440140. The second author is grateful for the support provided by the NSF through Grant DMS-1708493 and by the University of Oklahoma Research Council.

\section{Notation and preliminaries}\label{sec:preliminaries}

\subsection{Notation and conventions}
Let $M$ be a smooth, simply connected, closed manifold with a smooth, almost effective cohomogeneity one action by a connected Lie group $G$. The orbit space of this action on a compact manifold $M$ is either a circle $\s^1$ or the interval $[0,1]$. The former case cannot happen if $M$ is simply connected, so we always assume $M/G = [0,1]$.

It is well known that such a cohomogeniety one action is characterized up to equivariant diffeomorphism by its group diagram $H \subseteq K^\pm \subseteq G$, where $G/H$ is a principal orbit, and $G/K^\pm$ denote the two singular orbits.  Following Frank \cite{Frank13}, we always assume the action is minimal in the sense that no normal subgroup of $G$ acts orbit equivalently. Note that passing to such a subaction for a non-minimal action would not change the diffeomorphism type of the manifold. Moreover, by replacing $G$ by a finite cover, we will always assume $G$ has the form $G = G'\times \gT^m$ for a connected, simply connected compact Lie group $G'$.

It follows from the general structure theory of cohomogeneity one manifolds that $M$ has a double disc bundle decomposition, i.e., $M$ is the union of two disc bundles $D(G/K^\pm)$ over the singular orbits $G/K^\pm$. The common boundary of these disc bundles has the topology of the principal orbits $G/H$. It also follows that the spaces $K^\pm/H$ are spheres of dimension one smaller than the codimension $k_\pm$ of the singular orbits $G/K^\pm \subseteq M$.

\subsection{Positive Euler characteristic}\label{sec:poseul}
Applying the Mayer--Vietoris sequence to the doublde disc bundle decomposition of $M$, we have the following relation:
	\[\chi(M) = \chi(G/K^+) + \chi(G/K^-) - \chi(G/H).\]
When $M$ has positive Euler characteristic and hence even dimension, $G/H$ has odd dimension $\dim(M) - 1$ and hence vanishing Euler characteristic. Moreover, $\chi(G/K^\pm)$ is non-negative and is positive precisely when $K^\pm$ has rank equal to $\rank(G)$. Consequently $\chi(M) > 0$ only if at least one of the $K^\pm$ has rank equal to $\rank(G)$. Since each $K^\pm / H$ is a sphere, it follows from the classification of homogeneous spheres (see Section \ref{transsphere}) that $H \subseteq K^\pm$ has corank one or zero. Putting these considerations together, we have the following bounds on the ranks of groups in the group diagram:
	\begin{equation}\label{eqn:ranks}
	\rank(G) - 1 = \rank(H) \leq \rank(K^\pm) \leq \rank(G).
	\end{equation}

\subsection{Transitive actions on spheres}\label{transsphere}

We recall the well known results of Borel, Montgomery, and Samelson.

\begin{theorem}\label{thm:transsphere}
Suppose $K$ is a connected, compact Lie group which acts effectively and transitively on a sphere $\s^m$.  If $H$ denotes the isotropy subgroup, then Table \ref{table:transsphere} summarizes the possibilities for the pair $(K,H)$.
	\begin{center}
	\begin{table}[ht]
	\begin{tabular}{|c|c|c|}
	\hline
	$K$ & $H$ & $m$\\\hline
	$\SO(n+1)$ & $\SO(n)$ & $n$\\
	$\SU(n+1)$ & $\SU(n)$ & $2n+1$\\
	$\Un(n+1)$ & $\Un(n)$ & $2n+1$\\
	$\Sp(n+1)$ & $\Sp(n)$ & $4n + 3$\\
	$\Sp(n+1)\times \sone$ & $\Sp(n)\times \sone$ & $4n+3$\\
	$\Sp(n+1)\times \Sp(1)$ & $\Sp(n)\times \Sp(1) $ & $4n+3$\\
	$\Gtwo$ & $\SU(3)$ & $6$\\
	$\Spin(7)$ & $\Gtwo$ & $7$\\
	$\Spin(9)$ & $\Spin(7)$ & $15$\\\hline
	\end{tabular}\caption{Effective transitive actions on spheres}\label{table:transsphere}
	\end{table}
	\end{center}
\end{theorem}

\subsection{Codimensions of singular orbits}
Note that $k_\pm$, the codimension of the singular orbit $G/K^\pm$, is even if and only if $\rank(K^\pm) = \rank(G)$. Indeed, this follows from Equation \ref{eqn:ranks} together with the fact that a Lie group's rank and dimension have the same parity. The following lemma follows from the double disc bundle decomposition and is proven in \cite{Hoelscher10,GroveWilkingZiller08}.

\begin{lemma}\label{lem:cohomtop}
There are no exceptional orbits, i.e., both of the codimensions $k_\pm \geq 2$. Moreover, we have the following:
	\begin{enumerate}
	\item If $k_\pm \geq 3$, then $G/K^\mp$ is simply connected and $K^\mp$ is connected.
	\item If both $k_\pm \geq 3$, then $G/H$ and $G$ are simply connected and $H$ is connected.
	\end{enumerate}
\end{lemma}

\section{Restrictions via rational homotopy theory}\label{sec:RHT}

The main results of this section provide proofs of Theorems \ref{thm:QP} and \ref{thm:4periodic} in three special cases corresponding to singular orbits of small codimension (see Propositions \ref{pro:bothsingnonorientable}, \ref{pro:onesingnonorientable}, and \ref{pro:s1s1}). Proposition \ref{pro:orbittype} is also proved and provides additional information about the codimensions $k_\pm$ of the singular orbits $G/K^\pm$ when $M$ is a $\QP^n_k$.

For a connected, nilpotent space $X$, we use the shorthand $\pi_m^\Q(X)$ to denote the homotopy group of degree $m$ of the rationalization $X_\Q$. For $m \geq 2$, $\pi_m^\Q(X)$ coincides with $\pi_m(X)\otimes \mathbb{Q}$. For $m = 1$, $\pi_1^\Q(X)$ is the Malcev completion of $\pi_1(X)$, a nilpotent group, and its rank is the sum of the ranks of the abelian groups $\Gamma_{i-1}/\Gamma_i$ in the central series
	\[\pi_1(X) = \Gamma_0 \supseteq \Gamma_1 \cdots \supseteq \Gamma_n = 1\]
where $\Gamma_i = [\Gamma, \Gamma_i]$. For our purposes, it will suffice to deal with the following two cases:
	\begin{itemize}
	\item $\pi_1(X)$ is finite, in which case $\pi_1^\Q(X) = 0$, and
	\item $\pi_1(X)$ is abelian and finitely generated, in which case $\pi_1^\Q(X) \cong \pi_1(X) \tensor \Q$.
	\end{itemize}
We will also use the notation $\pi_{odd}^\Q(X)$ for $\bigoplus_{m=0}^\infty \pi_{2m+1}^\Q(X)$ and similarly for $\pi_{even}^\Q(X)$. Recall that a nilpotent space $X$ is rationally elliptic if both $H^\ast(X;\Q)$ and $\pi_\ast^\Q(X)$ are finite dimensional vector spaces.

\subsection{Grove-Halperin model of cohomogeneity one manifolds}
To each nilpotent space $X$, we may associate a commutative graded differential algebra, the minimal Sullivan algebra, which characterizes the rational homotopy type of $X$. Grove and Halperin \cite{GroveHalperin87} proved that, like homogeneous spaces and biquotients, cohomogeneity one manifolds are also rationally elliptic.  They analyze the Sullivan minimal model of a cohomogeneity one manifold with group diagram $H \subseteq K^\pm \subseteq G$, and prove the following.

\begin{theorem}[Grove--Halperin]\label{GHtype}
Let $\mathcal F$ denote the homotopy fiber of the inclusion $G/H \to M$ of the principal orbit of a cohomogeneity one manifold of positive Euler characteristic. There is a finite cover $\overline{\mathcal F} \to \mathcal F$ such that
	\[\overline{\mathcal F} \simeq_\Q 
		\left\{ \begin{array}{lcl}
		\s^3 \times \s^3 \times \Omega \s^7						&\mathrm{if}&	h = 2\\
		\s^1 \times \s^{2k_- - 1} \times \Omega \s^{2k_- + 1}			&\mathrm{if}&	h=1\\
		\s^{k_+-1} \times \s^{k_- -1} \times \Omega \s^{k_+ + k_- -1}	&\mathrm{if}& h = 0
		\end{array}\right.\]
where $h\in\{0,1,2\}$ is the number of non-orientable singular orbits. Moveover, $\overline{\mathcal F} = \mathcal F$ if $h = 0$, $\overline{\mathcal F} \to \mathcal F$ is two-to-one if $h = 1$, and $\overline{\mathcal F} \to \mathcal F$ is the universal cover and $\pi_1(\mathcal F)$ is the quaternion group $Q_8$ if $h = 2$.  Finally, if $h=1$, then $k_-$ must be even.
\end{theorem}

\begin{remark}
Grove and Halperin \cite[Proposition 3.5]{GroveHalperin87} use the terminology \textit{twisted} to describe the case where the bundle $\s^{\ell_\pm}\rightarrow G/H\rightarrow G/K^\pm$ is non-orientable.  In our case, $G/H$, being a codimension $1$ submanifold of the simply connected manifold $M$, is orientable.  It follows that the bundle is non-orientable iff $G/K^\pm$ is non-orientable.
\end{remark}

\subsection{Computation of the connecting homomorphism}
In this subsection, we record some general computations about the connecting homomorphism in the long exact sequence in rational homotopy associated to the fibration $G/H \to M$ with homotopy fiber $\mathcal F$. As in the previous section, $M$ denotes a cohomogeneity one manifold and we keep the notation from the previous section. The main results are Lemma \ref{lem:dimfailure} and Proposition \ref{connecthom}. The combination of the Grove-Halperin model and these results imply strong restrictions on the rational homotopy of an even-dimensional cohomogeneity one manifold with singly generated or four-periodic rational cohomology. Our main tool is a lemma concerning the quantity $d(X)$ defined as follows:

\begin{definition}\label{def:d}
For a connected, nilpotent space $X$ with $\dim \pi_*^\Q(X) < \infty$, set
	\[d(X) = \sum \deg a_i - \sum (\deg b_j - 1) = \sum_{\mathrm{odd}~k\geq 1} k \of{\dim \pi_{k}^\Q(X) - \dim \pi_{k+1}^\Q(X)},\]
where the $a_i$ and $b_i$ are any choice of graded basis of the odd and even degree rational homotopy groups of $X$.
\end{definition}

Note that $d(X)$ measures the (cohomological) dimension of $X$ when $X$ is rationally elliptic (see \cite[Theorem 32.15]{FelixHalperinThomas01}). In particular, since cohomogeneity one manifolds are rationally elliptic (see \cite{GroveHalperin87}), it follows that $d(M) = d(G/H) + 1$ for a cohomogeneity one $G$-manifold $M$ with principal isotropy $H$.

Note also that $d(\Omega \s^n) = 2-n$, independent of whether $n$ is even or odd. In particular, $d(\mathcal F) = d(\overline{\mathcal F}) = 1$ for each of the  homotopy fibers $\mathcal F$ in Theorem \ref{GHtype}.

In addition to these remarks, we need the following basic lemma.

\begin{lemma}\label{lem:dimfailure}
Suppose $F\xrightarrow{i} E\xrightarrow{p} B$ is a fibration for which $d$ is defined on all three spaces.  Let $\partial_k:\pi_{k}^\Q(B)\rightarrow \pi_{k-1}^\Q(F)$ denote the connecting homomorphism in the long exact sequence of rational homotopy groups associated to the fibration.  Then 
	\[d(E)  = d(F) + d(B) - 2 \sum_{\mathrm{odd}~k \geq 1} \dim \operatorname{im} \partial_k.\] 
\end{lemma}

\begin{proof}
For fixed $k$, the long exact sequence in rational homotopy groups associated to the fibration $F\rightarrow E\rightarrow B$ gives isomorphisms $\pi_k^\Q(F) \cong \operatorname{im}\partial_{k+1} \oplus \ \ker \pi_k(p)$, $\pi_k^\Q(E) \cong \ker \pi_k(p) \oplus \ker \partial_k,$ and $ \pi_k^\Q(B) \cong \ker \partial_k \oplus \operatorname{im}\partial_k.$  Thus, \begin{equation}\label{dcalc}\dim \pi_k^\Q(E) = \dim \pi_k^\Q(F) + \dim \pi_k^\Q(B) - \dim \operatorname{im}(\partial_{k+1}) - \dim \operatorname{im}(\partial_k).\end{equation} 
For odd $k\geq 1$, the contribution of Equation $\eqref{dcalc}$ in degree $k$ and $k+1$ to $d(E)$ is \begin{align*} k(\dim\pi_k^\Q(E) -\dim\pi_{k+1}^\Q(E)) =& k(\dim\pi_k^\Q(F) -\dim\pi_{k+1}^\Q(F))\\  &+k(\dim\pi_k^\Q(B) -\dim\pi_{k+1}^\Q(B))\\ &-k(\dim \operatorname{im}(\partial_k) - \dim \operatorname{im}(\partial_{k+2})).  \end{align*}  The result follows by summing both sides over odd $k \geq 1$.
\end{proof}

Combining this lemma with the above remarks, we have the following corollaries, the first of which is implied by \cite[Theorem 1.4.iii]{Halperin78}, and the second of which is closely related to \cite[Lemma 6.3]{GroveHalperin87}.

\begin{corollary}\label{fbundle}
For a fiber bundle $F\rightarrow E\rightarrow B$ of rationally elliptic spaces, the map $\partial_\ast:\pi_{odd}^\Q(B)\rightarrow \pi_{even}^\Q(F)$ is zero and $d(E) = d(B) + d(F)$. In particular, the map $\pi_{odd}^\Q(E)\rightarrow \pi_{odd}^\Q(B)$ is surjective and the map $\pi_{even}^\Q(F)\rightarrow \pi_{even}^\Q(E)$ is injective.
\end{corollary}

\begin{corollary}\label{cor:piodd}
For the fibration $\mathcal F\rightarrow G/H\rightarrow M$ coming from a cohomogeneity one $G$-manifold $M$, the map 
	$\partial_\ast:\pi_{odd}^\Q(M)\rightarrow \pi_{even}^\Q(\mathcal{F})$ has rank one.
\end{corollary}

For our purposes, we need a stronger version of Corollary \ref{cor:piodd}:

\begin{proposition}\label{connecthom} 
For a cohomogeneity one manifold as in Theorem \ref{GHtype}, the connecting map $\partial:\pi_{2m+1}^\Q(M)\rightarrow \pi_{2m}^\Q(\mathcal{F})$ is non-trivial precisely when $\pi_{2m}^\Q(-)$ of the loop space factor of $\overline{\mathcal F}$ is non-trivial.
\end{proposition}

\begin{proof}
We start with some general remarks. Set $E = G/H$. Let $(\Lambda V_M, d_M))$ and $(\Lambda V_{\mathcal F}, d_{\mathcal F})$ be minimal Sullivan models for $M$ and $\mathcal F$, respectively. Since $\mathcal F \to E \to M$ is a fibration with $\pi_1(M) = 0$, we can form a relative Sullivan algebra $(\Lambda(V_M \oplus V_F), d_E)$ for $E$.  The differential $d_E$ has the property that $d_E(x) = d_M(x)$ for $x \in V_M$ and $d_E(y) - d_F(y) \in \Lambda^{\geq 1}V_M \otimes \Lambda V_F$. Moreover, the composition
	\[V_F^k \subseteq \Lambda V_M \tensor \Lambda V_F \stackrel{d_E}{\longrightarrow} \Lambda V_M \tensor \Lambda V_F \to V_M^{k+1},\]
denoted $d_0$, is dual to the connecting homomorphism $\partial: \pi_{k+1}^\Q(M) \to \pi_k^\Q(\mathcal{F})$ (see \cite[Proposition 15.13]{FelixHalperinThomas01}). 

We proceed to the proof. It suffices by Corollary \ref{cor:piodd} to prove that $\partial = 0$ in degree $k+1$ when $k$ is even and $\pi_k^\Q(-)$ of the loop space factor of $\overline{\mathcal F}$ is zero. By the remarks above and the Grove-Halperin structure theorem for the rational homotopy type of $\overline{\mathcal F}$, it suffices to prove that $d_0 = 0$ on even degrees $k$ for which there is a spherical factor of $\overline{\mathcal F}$. To do this, suppose $y_k \in V_F^k$ is non-zero such that $y_k^2 = d_F(y_{2k-1})$ for some $y_{2k-1} \in V_F^{2k-1}$. First, note that

\[d_E(y_{2k-1}) = d_F(y_{2k-1}) + d_0(y_{2k-1}) +g_1 = y_k^2 + d_0(y_{2k-1}) + g_1\]
for some $g_1 \in \Lambda^{\geq 2}(V_M \oplus V_F)$. Applying $d_E$ again and using the properties of $d_E$ when restricted to $V_M$ or $V_F$, we have
\[0 = d_E^2(y_{2k-1}) = 2 y_k \of{0 + d_0(y_k) + g_2} + d_M(d_0(y_{2k-1})) + d_E(g_1)\]
for some $g_2 \in \Lambda^{\geq 2}(V_M \oplus V_F)$. Extracting the part of this expression in $V_M \tensor V_F \subseteq \Lambda V_M \tensor \Lambda V_F$, we obtain $0 = 2 y_k d_0(y_k)$ and hence $d_0(y_k) = 0$, as claimed.
\end{proof}

\subsection{Applications to $\QP^n_k$ and rationally four-periodic manifolds}\label{sec:applications}
Here we apply Proposition \ref{connecthom} to analyze cases in Theorems \ref{thm:QP} and \ref{thm:4periodic} involving small codimensions $k_\pm$. We start by recording the minimal Sullivan model associated to a $\QP^n_k$.

\begin{proposition}\label{pro:modelQP}
A simply connected manifold $M$ is a $\QP^n_k$ if and only if one of the following hold:
	\begin{enumerate}
	\item $k$ is odd, $n = 1$, $\pi_k^\Q(M) \cong \Q$, and all other rational homotopy groups vanish.
	\item $k$ is even, $n \geq 1$, $\pi_k^\Q(M)\cong \pi_{k(n+1)-1}^\Q(M) \cong \Q$, and all other rational homotopy groups vanish.
	\end{enumerate}
Note that $\chi(M) = 0$ in the first case and that $\chi(M) > 0$ in the second. In both cases, $M$ is rationally elliptic and there is exactly one rational homotopy type for fixed $(k,n)$.
\end{proposition}

Similarly it is straightforward to characterize rationally elliptic, simply connected, closed manifolds with four-periodic rational cohomology (see, for example, \cite[Theorem 1.1]{DeVito18}, for a proof). Following \cite{Wilking03,Kennard13}, we say that an orientable closed manifold $M$ has four-periodic rational cohomology if there exists $x \in H^4(M;\Q)$ such that the maps $H^i(M;\Q) \to H^{i+4}(M;\Q)$ induced by multiplication by $x$ are surjections for $0 \leq i < \dim M - 4$ and injections for $0 < i \leq \dim M - 4$. 

\begin{proposition}\label{pro:model4periodic}
For a simply connected, closed manifold $M$ with four-periodic rational cohomology and even dimension, the following are equivalent:
	\begin{enumerate}
	\item $M$ is rationally elliptic.
	\item $M$ is a $\QP^1_k$, $\QP^n_2$, or $\QP^n_4$ or a rational $\s^2 \times \HP^n$ or $\s^3 \times \s^3$.
	\end{enumerate}
In both cases, $M$ has positive Euler characteristic unless $M \simeq_\Q \s^3 \times \s^3$.
\end{proposition}

We now apply Proposition \ref{connecthom} to prove the main results of this section: Propositions \ref{pro:bothsingnonorientable}--\ref{pro:orbittype}. The first two of these will exclude the possibility of non-orientable singular orbits in the proofs of the main theorems. The latter two provide information about the codimensions $k_\pm$ of the singular orbits in the case where they are both orientable.

\begin{proposition}\label{pro:bothsingnonorientable}
Let $M$ be an even-dimensional, simply connected, closed cohomogeneity one manifold with both singular orbits non-orientable. 
	\begin{enumerate}
	\item If $M$ is a $\QP^n_k$ for some even $k$, then $M$ is equivariantly diffeomorphic to $\s^4$ with $G = \SO(3)$ acting by its unique irreducible $5$-dimensional representation.  
	\item If $M$ is a rational $\s^2 \times \HP^n$, then $M$ is diffeomorphic to $\s^2\times \s^4$.
	\end{enumerate}
\end{proposition}

Note that the first part of the proposition when $k = 2$, $k = 4$, or $(k,n) = (8,2)$ follows from \cite[Section 6]{Uchida77}, \cite[Section 6]{Iwata78}, and \cite[Proposition 1]{Iwata81}, respectively.

\begin{proof}
If both singular orbits are non-orientable, then Theorem \ref{GHtype} implies that the contribution of the loop space factor of $\overline{\mathcal F}$ to the rational homotopy in even degrees occurs in degree six. By Proposition \ref{connecthom}, it follows that $\pi_7^\Q(M)\neq 0$. Restricting to the case of a $\QP^n_k$ or a rational $\s^2 \times \HP^n$, the models computed in Propositions \ref{pro:modelQP} and \ref{pro:model4periodic} imply that $M$ has the rational homotopy type of $\s^4$, $\CP^3$, or $\s^2\times \s^4$.

The case where $M\simeq_\Q \CP^3$ does not arise since a rational $\CP^3$ does not admit a cohomogeneity one action two non-orientable singular orbits (see Uchida \cite{Uchida77}).

If $M\simeq_\Q \s^4\cong \HP^1$, then the result follows from the classification of Iwata \cite{Iwata78} or from Hoelscher \cite[Section 1.7.3]{Hoelscher10}.

Finally, suppose $M\simeq_\Q \s^2\times \s^4$.  Considering the fibration $\mathcal{F}\rightarrow G/H\rightarrow M$, we see that if $\pi_4^\Q(M)\rightarrow \pi_3^Q(\mathcal{F})$ is the zero map, then $\pi_4^\Q(G/H)$ is the highest non-trivial rational homotopy group.  Computing a Sullivan model then shows that $G/H$ has unbounded rational cohomology giving a contradiction.  It follows that $\pi_\ast^\Q(G/H)\cong \pi_\ast^\Q(\s^2\times \s^3)$. From, e.g. \cite[Theorem 3.1, Case 1]{DeVito17}, since $G$ and $H$ share no common normal subgroups, it follows that $G\cong \SU(2)\times \SU(2)$ and $H_0 \cong \sone$.  Since both $k_\pm = 2$, $K^\pm_0 = \gT^2$.  This situation is studied by Hoelscher \cite[Section 3.2.1, Case A]{Hoelscher10}.  Hoelscher shows that, since $G/K^\pm$ are non-orientable, both $K^\pm$ are disconnected, the action must be a product action.  Hence, $M$ is diffeomorphic to $\s^2\times \s^4$.
\end{proof}

\begin{proposition}\label{pro:onesingnonorientable}
Let $M$ be an even-dimensional, simply connected, closed cohomogeneity one manifold with exactly one non-orientable orbit.  Then $M$ cannot be a rational $\s^2\times \mathbb{H}P^n$.  In addition, if $M$ is a $\QP^n_k$, then $M$ is equivariantly diffeomorphic to $\CP^{2n}$  with a linear action by $\SO(2n+1) \subseteq \Un(2n+1)$ or $\Gtwo\subseteq \Un(7)$, where $2n = 6$ in the latter case.
\end{proposition}

Note that the case of a $\QP^n_k$ with $k = 2$, $k = 4$, or $(k,n) = (8,2)$ is proved in \cite[Section 7]{Uchida77}, \cite[Theorem 2.1.4, Part B]{Iwata78}, and \cite[Proposition 1]{Iwata81}, respectively.

\begin{proof}
From Theorem \ref{GHtype}, we know that, up to interchanging $k_+$ and $k_-$, $\mathcal{F}$ is rationally $\s^1 \times \s^{2k_- - 1} \times \Omega \s^{2k_- + 1}$, where $k_-$ is even.  By Proposition \ref{connecthom}, $\pi_{2k_- + 1}^\Q(M) \neq 0$.

We first rule out the possibility that $M\simeq_\Q \s^2\times\HP^n$. There are two non-trivial rational homotopy groups in odd degrees, so either $2k_- + 1 = 3$ or $2k_- + 1 = 4n + 3$. Both cases imply that $k_-$ is odd, a contradiction.

This leaves the case where $M$ is a $\QP^n_k$, which has non-trivial rational homotopy groups in degrees $k$ and $k(n+1) - 1$. Hence $2k_- + 1 = k(n+1) - 1$. Now consider the connecting map $\pi_k^\Q(M) \cong \Q\rightarrow \pi_{k-1}^\Q(\mathcal{F})$.  Assume initially this connecting map is the zero map.  Then from the relative Sullivan model of $G/H$, we see it has a generators in degree $1, k, $ and $k(n+1)-3$.  If $k > 2$, then no differential can kill powers of $k$ and $G/H$ has unbounded rational cohomology, a contradiction. If instead the connecting map is non-zero, we have $ k = 2$ or $k = 2k_-$.  In the latter case, $k(n+1) - 2 = k$ and it again follows that $k = 2$. The equivariant diffeomorphism rigidity now follows from Uchida (see \cite[Proposition 7.1.3]{Uchida77}).
\end{proof}

\begin{proposition}\label{pro:s1s1}
Let $M$ be an even-dimensional, simply connected, closed cohomogeneity one manifold with both singular orbits orientable and of codimension $k_\pm = 2$.
	\begin{enumerate}
	\item If $M$ is a $\QP_n^k$, then $M$ is equivariantly diffeomorphic to $\s^2$ with $\SO(2)$ acting by rotations.
	\item If $M \simeq_\Q \s^2 \times \HP^n$, then $M$ is diffeomorphic to $\s^2\times \HP^n$ or $\s^2\times \Gtwo/\SO(4)$.
	\end{enumerate}
\end{proposition}

\begin{proof}
From Theorem \ref{GHtype}, we know that $\mathcal{F}$ has the rational homotopy type of $\s^1\times\s^1\times\Omega\s^3$. Hence Proposition \ref{connecthom} implies that $\pi_3^\Q(M) \neq 0$.

If $M$ is a $\QP_k^n$, this implies that $M$ is a rational $\s^2$ and the result follows from the low dimensional classification of cohomogeneity one manifolds \cite{Hoelscher10}.

Suppose then that $M$ is a rational $\s^2\times \HP^n$.  In the fibration $\mathcal{F}\rightarrow G/H\rightarrow M$, we note that the connecting homomorphism $\pi_2^\Q(M)\rightarrow \pi_1^\Q(\mathcal{F})$ must be non-trivial, for otherwise, the Sullivan model of $G/H$ shows $H^\ast(G/H;\Q)$ is unbounded, giving a contradiction.  It follows that $G/H$ has the same rational homotopy groups as $\s^1\times \HP^n$.  Since $G/H$ is nilpotent \cite{GroveHalperin87}, it now follows that $G/H$ has the rational homotopy type of $\s^1\times \HP^n$.

Because $\pi_1^\Q(G/H) \cong \Q$, $G = G'\times \sone$ where the projection of $H$ into the $\sone$ factor has finite image.  Then $H_0\subseteq G'$, the identity component of $H$, has maximal rank and so $G'/H_0$ is a simply connected $\QP_4^n$.  So, from, e.g. \cite{KapovitchZiller04}, $G'/H_0$ is either $\HP^n$ or $\Gtwo/\SO(4)$.

From \cite[Example 4L4, pg. 493]{Hatcher01}, every map of $\HP^n$ to itself has a fixed point if $n > 1$.  Likewise, every map from $\Gtwo/\SO(4)$ to itself has a fixed point, as follows from the Lefschetz theorem with rational coefficients.  By looking at the deck group action, it follows that $\HP^n$ for $n>1$ and $\Gtwo/\SO(4)$ are not the total space of any non-trivial covering.  Thus, it follows in these cases that the projection $\pi:G\rightarrow G$ has $\pi(H)=H$, for otherwise $G'/H_0\rightarrow G'/\pi(H)$ is a non-trivial covering.  So, $H = H_0\times \Z_m$ with $\Z_m\subseteq \sone$.

In the case $n=1$, $\HP^1 = \s^4$, there are of course fixed point free diffeomorphisms.  However, such a map must be orientation reversing.  It follows in this case that $\pi(H)/H_0$, if non-trivial, acts on $G'/H_0$ in an orientation reversing manner.  In particular, this implies $G/H$  is non-orientable.  Since $G/H$ is a codimension one submanifold of a simply connected manifold, this cannot happen.  It follows that $H = H_0\times \Z_m$ with $\Z_m\subseteq \sone$ in this case as well.

Now, by dividing out the common normal subgroup $\Z_m\subseteq \sone$ from $G$ and $H$, we may assume $H = H_0$ is connected.  And thus, that $G/H = \s^1\times \HP^n$ or $\s^1\times (\Gtwo/\SO(4))$.  Now, in each case, $H_0\subseteq G'$ is maximal.  It follows that both $K^\pm = H_0\times \sone$.  Thus, we have the group diagram of a product action on $\s^2\times \HP^n$ or $\s^2\times (\Gtwo/\SO(4))$.
\end{proof}

The fourth and final application of Proposition \ref{connecthom} we prove here applies only to $\QP^n_k$. It provides detailed information about the codimensions $k_\pm$ of the singular orbits $G/K^\pm$ and about the rational homotopy of the principal orbit $G/H$.

\begin{proposition}\label{pro:orbittype}
Let $M^{kn}$ be an even-dimensional $\QP^n_k$ that admits a cohomogeneity one action with group diagram $H \subseteq K^\pm \subseteq G$ and orientable singular orbits of codimensions $k_\pm = \cod(G/K^\pm)$. Assume $k_+$ is even.
\begin{enumerate}
	\item If $k_-$ is even, then $k_+$ and $k_-$ are divisible by $k$ and sum to $k(n+1)$. Moreover, $G/K^\pm \simeq_\Q \QP^{\frac{k_\mp}{k} - 1}$ and $\pi_*^\Q(G/H)$ has dimension three and is generated in degrees $k$, $k_+ - 1$, and $k_- - 1$.
	\item If $k_-$ is odd, then $k_+ + k_- = \frac{n+1}{2}k+1$ and one of the following occurs:
		\begin{enumerate}
			\item $\pi_\ast^\Q(G/H)$ has dimension five  with generators in odd degrees $k_+-1, 2k_- - 3,$ and $\frac{n+1}{2}k-1$ and in even degrees $k_- - 1$ and $k$. Moreover $n \geq 2$.
			\item $\pi_\ast^\Q(G/H)$ has dimension three with generators in degrees in $\{k_+ - 1, 2k_- - 3, \frac{n+1}{2}k-1, k_- -1\} \setminus\{k-1\}$ for some $k-1 \in\{k_+ -1, 2k_- -3, \frac{n+1}{2}k-1\}$.
			\end{enumerate}
\end{enumerate}
\end{proposition}

Note that Case 1 includes the possibility that the action has a fixed point (i.e., if one of the $k_\pm = kn$). In general, a cohomogeneity one manifold with a fixed point is equivariantly diffeomorphic to a CROSS with linear action (see Hoelscher \cite{Hoelscher10}). 

\begin{proof}
Suppose first that $k_-$ is even. Then $\pi_i^\Q(\mathcal{F})$ is $\Q$ in odd degrees $k_+ - 1$ and $k_- - 1$ and in even degree $k_+ + k_- - 2$. By Lemma \ref{connecthom}, this even degree equals $k(n+1)-2$, so $k_+ + k_- = k(n+1)$. Note, in particular, that if $k_+ \leq k$ (resp. $k_- \leq k$), then $k_-$ (resp. $k_+$) is at least, and hence equals, the dimension of $M$. In other words, the action has a fixed point. Since this possibility is in the conclusion of the lemma, we may assume therefore that both $k_\pm > k$. In particular, $\pi_{k-1}^\Q(\mathcal{F}) = 0$ and the map $\partial_k = 0$. It follows that $\pi_*^\Q(G/H)$ is generated in degrees $k_+ - 1$, $k_- -1$, and $k$. Since $G/H$ has finite dimension, at least one of the $k_\pm$ is divisible by $k$. By the equation $k_+ + k_- = k(n+1)$, they are both divisible by $k$. It follows that $G/H$ has the rational homotopy type of $\s^{k_+ - 1} \times \QP_k^{\frac{k_-}{k} - 1}$ or $\s^{k_- - 1} \times \QP_k^{\frac{k_+}{k} - 1}$. In either case, an analysis of the fibrations $K^\pm/H \to G/H \to G/K^\pm$ imply that $G/K^\pm \simeq_\Q \QP^{\frac{k_\mp}{k} - 1}$.

Next, suppose that $k_-$ is odd.  Then $\pi_*^\Q(\mathcal{F})$ is five-dimensional with odd generators in degrees $k_+ - 1$, $2k_- - 3$, and $k_+ + k_- - 2$ and with even generators in degrees $k_- - 1$ and $2k_+ + 2k_- - 4$. Lemma \ref{connecthom} implies that $k(n+1) = 2k_+ + 2k_- - 2$. The computation of $\pi_*^\Q(G/H)$ now follows by applying the long exact sequence in rational homotopy groups to the fibration $G/H \to M$. Note that Case 2.a holds if all of the connecting homomorphisms $\partial_k : \pi_k^\Q(M) \to \pi_{k-1}^\Q(\mathcal{F})$ are all trivial and that otherwise Case 2.b occurs. 
	
Finally, note in Case 2 that if $n=1$, then $2 = \chi(M) = \chi(G/K^+)$.  Since $G/K^+$ is simply connected, it must be a rational sphere.  In particular, $\dim\pi_\ast^\Q(G/K^+) = 2$.  From the bundle $\s^{k_+-1}\rightarrow G/H\rightarrow G/K^+$, we see that $\dim \pi_\ast^\Q(G/H) = 3$.
\end{proof}

\section{The non-primitive case}\label{sec:non-primitive}

A cohomogeneity one action of a compact Lie group $G$ on a closed manifold $M$ is called \textit{non-primitive} if for a group diagram $H\subseteq K^\pm \subseteq G$ there exists a subgroup $L\subsetneq G$ containing both $K^\pm$. It is called \textit{primitive} otherwise.  When the action is non-primitive, there is a bundle $M_L\rightarrow M\rightarrow G/L$, where $M_L$ denotes the cohomogeneity one manifold with group diagram $H\subseteq K^\pm \subseteq L$.  This bundle is associated to the $L$-principal bundle $L\rightarrow G\rightarrow G/L$, so $M$ is diffeomorphic to $M_L\times_L G/L$. Note in particular that, unlike generalizing from minimal to not necessarily minimal actions, potentially new diffeomorphism types arise when generalizing from the primitive to the not necessarily primitive case.

Before proceeding to the proofs of Theorems \ref{thm:QP} and \ref{thm:4periodic}, we prove two lemmas. The first will imply that, for our purposes, we may assume $L$ is connected.

\begin{proposition}\label{Lconnect}
Suppose $M$ is a cohomogeneity one manifold with group diagram $H\subseteq K^\pm \subseteq G$ with at least one of $K^\pm$ connected.  If $M$ is not primitive, then there is a connected $L\subsetneq G$ with $K^\pm\subseteq L$.
\end{proposition}

\begin{proof}  
Assume $K^+$ is connected.  Because $M$ is not primitive, there is a possibly disconnected Lie group $L\subsetneq G$ with both $K^\pm\subseteq L$.  We will show that both $K^\pm \subseteq L_0$, the identity component of $L$.

Since $K^+$ is connected, $K^+\subseteq L_0$.  For $K^-$, we first note that $K^-/H$ is a sphere of dimension $k_- - 1 \geq 1$ and hence is connected. This implies that any $g \in K^-$ is connected by a path in $K^- \subseteq L$ to a point in $g' \in H$. Since $H \subseteq K^+ \subseteq L_0$, there is a path in $L_0$ connecting $g'$ to the identity. Concatenating, we obtain a path in $L$ from the identity to $g$. This proves that $K^- \subseteq L_0$, as claimed.
\end{proof}

The second lemma provides an obstruction to the existence of non-trivial bundles over the homogeneous space $G/L$ when the topology of the base is simple and the dimension of the fiber is small.

\begin{lemma}\label{lem:SmallFiber}
Suppose $L\subseteq G$ are compact Lie groups and that $F$ is a compact $L$-manifold. If $G/L = \s^{2l}$ and $2l \geq \max(8, \dim F + 2)$, then the associated bundle $G\times_L F \to G/L$ is trivial.  Similarly, if $G/L = \F_4/\Spin(9) = \OP^2$ and $\dim F \leq 7$, then the associated bundle $G\times_L F\to G/L$ is trivial. 
\end{lemma}
	
\begin{proof}
We only prove the first statement; the proof of the second is analogous.

Because $G/L$ has positive Euler characteristic, $L$ has full rank in $G$.  Further, since $G$ acts transitively on a sphere of dimension bigger than $6$, it follows from Theorem \ref{thm:transsphere} that, up to cover, $G = G'\times \Spin(2l+1)$ and $L = G'\times \Spin(2l)$ with each factor of $L$ embedded into the corresponding factor of $G$.
	
Because $L$ is compact, we may assume the $L$ action on $F$ is isometric.  This action is characterized by a homomorphism $L = G'\times \Spin(2l)\rightarrow \Isom(M_L)$.  Since $F$ has dimension at most $2l-2$, its isometry group $\Isom(M_L)$ has dimension less than $\dim \Spin(2l)$.  As $\Spin(2l)$ is simple (since $2l > 4$), it follows that the map $L\rightarrow \Isom(M_L)$ is trivial on the $\Spin(2l)$ factor. Therefore
	\[G \times_L F =  (G'\times \Spin(2l+1))\times_{G'\times \Spin(2l)} F \cong  \s^{2l}\times (G'\times_{G'} F) \cong \s^{2l} \times F.\]
\end{proof}

We can now prove Theorems \ref{thm:QP} and \ref{thm:4periodic} in the non-primitive case.

\subsection{Non-primitive actions on $\QP^n_k$}\label{sec:nonprimitiveQP}

In this subsection, we prove Theorem \ref{thm:QP} in the non-primitive case. We are given an even-dimensional, simply connected, closed manifold $M$ with singly generated cohomology (i.e., an even-dimensional $\QP^n_k$), and we assume it admits a cohomogeneity one action with diagram $H \subseteq K^\pm \subseteq G$. If the parameters $n$ and $k$ are standard (i.e., either $k \in \{2,4\}$ or $k = 8$ and $n = 2$), then the result follows by Uchida and Iwata's classifications. Our task is to exclude the possibility that $k \geq 6$ if the action is non-primitive.

We proceed by contradiction. Assume that $k \geq 6$. From the remarks at the end of Section \ref{sec:RHT}, we may assume that both singular orbits are orientable and that the codimension $k_- \geq 3$. In particular, $G/K^+$ is simply connected and $K^+$ is connected. Non-primitivity implies the existence of a fibration $M \to G/L$ with fiber $M_L$, which is the cohomogeneity one manifold with group diagram $H \subseteq K^\pm \subseteq L$. Moreover, $M = G \times_L M_L$.  Note that $L$ is connected by Proposition \ref{Lconnect} and hence that $M_L$ is connected.

Since $M$ is simply connected and $M_L$ is connected, $G/L$ is simply connected and we may apply the Serre spectral sequence. Note that $M_L$ and $G/L$ are rationally elliptic and have positive Euler characteristic by the formula $\chi(M) = \chi(G/L) \chi(M_L)$. Hence their rational cohomology vanishes in odd degrees, and the spectral sequence degenerates on the $E_2$ page. Furthermore, a consideration of the edge homomorphisms associated to this spectral sequence shows that $M_L$ has rational cohomology isomorphic to $\Q[x]/(x^{m+1})$ where $x$ has degree $k$ and that $G/L$ has rational cohomology isomorphic to $\Q[y]/(y^{\frac{n-m}{m+1}})$ where $y$ has degree $k(m+1)$. Since $G/L$ is simply connected, it is a $\QP^{\frac{n+1}{m+1} - 1}_{k(m+1)}$. Also note that $m \geq 1$ and $k \geq 6$, so the generator of $H^\ast(G/L;\Q)$ has degree at least $12$. We now obtain integral (in fact, diffeomorphism) rigidity by the classification of homogeneous spaces with singly generated rational cohomology (see \cite{KapovitchZiller04}). Specifically, we conclude that $G/L$ is a standard, even-dimensional sphere, say $\s^{2l}$. Hence $2l = k(m+1)$ and $1 = \frac{n+1}{m+1} - 1$. These imply the estimates required to apply Lemma \ref{lem:SmallFiber}. Hence the bundle $G \times_L M_L \to G/L$ is trivial and the total space $M = G \times_L M_L$ is diffeomorphic to the product $G/L \times M_L$. Since $M$ has singly generated rational cohomology, this is a contradiction.

\subsection{Non-primitive actions on $\s^2\times \HP^n$}\label{sec:nonprimitive4periodic}
In this subsection, we prove Theorem \ref{thm:4periodic} in the non-primitive case. We are given an even-dimensional, simply connected, closed manifold $M$ with four-periodic rational cohomology and positive Euler characteristic. By Proposition \ref{pro:model4periodic}, $M$ is a rational $\s^2 \times \HP^n$. We are also given a non-primitive cohomogeneity one action with diagram $H \subseteq K^\pm \subseteq G$, and a subgroup $L$ such that $K^\pm \subseteq L \subsetneq G$. 

By Propositions \ref{pro:bothsingnonorientable}, \ref{pro:onesingnonorientable}, and \ref{pro:s1s1}, we may assume that both singular orbits are orientable and that the codimension $k_- \geq 3$. In particular, $G/K^+$ is simply connected and $K^+$ is connected. As in the $\QP^n_k$ case, it follows that $L$ is connected and that there exists a fibration $M \to G/L$ with connected fiber $M_L$, which is the cohomogeneity one manifold with group diagram $H \subseteq K^\pm \subseteq L$, and that further, $M = G \times_L M_L$.

\begin{lemma}  
Suppose $F\rightarrow M\rightarrow B$ is a fiber bundle with connected fiber $F$ and one-connected base $B$. If $M\simeq_\Q \s^2\times \HP^n$, then either $(F,B)\simeq_\Q (\s^2 \times \HP^k, \QP_{4(k+1)}^{\frac{n-k}{k+1}})$ or $(F,B)\simeq_\Q (\HP^k, \s^2 \times \QP_{4(k+1)}^{\frac{n-k}{k+1}})$.
\end{lemma}

\begin{proof}
According to \cite{Halperin78}, $F$ and $B$ must be rationally elliptic.  From Corollary \ref{fbundle}, the map $\pi_{odd}^\Q(M)\rightarrow \pi_{odd}^\Q(B)$ is surjective.  Thus, the non-trivial odd dimensional rational homotopy groups of $B$ are in dimensions a subset of $\{3,4n+3\}$.  The subset must be non-empty because $F$ is rationally elliptic.  Assume initially that $\pi_3(B)\neq 0$.  Since $B$ is rationally elliptic with positive Euler characteristic, this implies $\pi_2(B)\neq 0$.  Because $F$ has positive Euler characteristic, $\pi_1^\Q(F) = 0$, so $\pi_2^\Q(M)\rightarrow \pi_2^\Q(B)$ is an isomorphism.

Now, since $\pi_{even}^\Q(F)$ injects into $\pi_{even}(M)$, $\pi_4^\Q(F)$ is the only non-trivial even degree rational homotopy group.  Since $F$ has positive Euler characteristic, $\pi_{4k+3}^\Q(F)\neq 0$ for precisely one $k\geq 1$.  It then follows in this case that $F$ is a rational $\HP^k$, from which is easily follows that $B$ is a rational $\s^2 \times \QP_{4(k+1)}^{\frac{n-k}{k+1}}$.

Next, assume $\pi_3^\Q(B) = 0$.  This implies $\pi_3(F) = \Q$, so $\pi_2^\Q(F) = \Q$ as well.  If $\pi_4(F)= 0$, then $F$ is rationally $\s^2$ so $B$ is rationally $\HP^n$.  Thus, assume $\pi_4(F) \neq 0$.  Since we know $\pi_{4n+3}^\Q(B) = \Q$ and that, apart from a single even dimension group, all other rational homotopy groups of $B$ vanish, it easily follows that
	\[(F,B)\simeq_\Q (\s^2 \times \HP^k, \QP_{4(k+1)}^{\frac{n-k}{k+1}}).\]
\end{proof}

Now, if the base $B = G/L$ is rationally $\QP_{4(k+1)}^{\frac{n-k}{k+1}}$, then $\dim B = 4(n-k) > 6$ for otherwise, $k+1 = n$, which contradicts the fact that $k+1|n-k$.  From Kapovitch-Ziller \cite{KapovitchZiller04}, it follows that $G/L$ is a standard even dimensional sphere or $G/L = \OP^2$.  Note also that if $G/L$ is a sphere, then $\dim B \geq \dim F + 2 = 4(k+1)$ because $k+1|n-k$.  Similarly, if $G/L  \cong \OP^2$, then $k = 1$, so $\dim F = 6$.  Thus, by Lemma \ref{lem:SmallFiber}, $M$ is diffeomorphic to $F\times B$ in this case.  In particular, $\dim \pi_{odd}^\Q(M) = 3$, a contradiction.

When the base is $\s^2\times \QP_{4(k+1)}^{\frac{n-k}{k+1}}$ with $n > k$, then it is still true that $(4(k+1), \frac{n-k}{k+1})$ is standard.  Indeed, $G$ cannot be simple by \cite{DeVito18}, so $G = G_1\times G_2$ and $L = L_1\times L_2$ with $L_i\subseteq G_i$ of full rank.  Without loss of generality, we assume $\pi_3^\Q(G_1/L_1)\neq 0$.  Since $\chi(G_1/L_1) > 0$, it follows that $G_1/L_1$ is rationally $\s^2$.  Now, by inspecting the rational homotopy groups of $G_2/L_2$, it follows that $G_2/L_2$ is a $\QP_{4(k+1)}^{\frac{n-k}{k+1}}$, so is standard by Kapovitch-Ziller.  In other words, $G_2/L_2$ is either a sphere $\s^{2n+2}$, $\HP^n$, $\Gtwo/\SO(4)$, or $\OP^2$.  Then the dimension of the fiber is $2n ,2,2, 6$ respectively.

For the first case, $G_2/L_2 = \s^{2n+2}$ with $k+1 = n-k$, implying $n$ is odd.  If $n = 1$, then $M$ is a linear $\s^2$-bundle over $\s^4$.  Non-trivial bundles over this form are rationally $\mathbb{C}P^3$, so this bundle must be trivial: $M$ is diffeomorphic to $\s^2\times \s^4$ in this case.  Thus, we may assume $n\geq 3$ is odd, so Lemma \ref{lem:SmallFiber} applies, giving the contradiction $\dim \pi_{odd}^\Q(M)\geq 3$.  Similarly, we achieve the same contradiction in the last case where $G_2/L_2\cong \OP^2$.

For the middle two cases, since $G_2/L_2$ is a rational $\HP^m$, it follows that the fiber is $\s^2$.  Now, $L_2$ splits as $\SU(2)\times L_2 '$ with $L_2'$ simple, so the $L_2$ action on the fiber $\s^2$ must be trivial because the only cohomogeneity one action on $\s^2$ is the standard circle action.  In particular, $\dim \pi_{odd}^\Q(M)\geq 3$ in this case as well, again giving a contradiction.

For the last case, we have $n = k$, so $F = M_L$ is a rational $\HP^n$ and $B = \s^2$.  From Iwata \cite{Iwata78}, we know that the $L$ action on $M_L$ is isometric with respect to the usual metric.  By clutching function arguments, such bundles over $\s^2$ are classified by $\pi_1(\Isom(M_L))$, the fundamental group of the isometry group of $M_L$.  When $M_L = \Gtwo/\SO(4)$, $\pi_1(\Isom(M_L)) = \{1\}$, so $M$ is diffeomorphic to $\s^2 \times (\Gtwo/\SO(4))$.  When $M_L = \HP^m$, $\pi_1(\Isom(M_L)) = \Z/2\Z$, so there precisely two such $M$ up to bundle equivalence.  In fact, as shown in \cite{DeVito18}, the total spaces of the two $\HP^m$ bundles over $\s^2$ have different Stiefel-Whitney classes, so are not even homotopy equivalent.

This concludes the proof of Theorem \ref{thm:4periodic} in the non-primitive case. We finish this section by noting that the three examples in the conclusion of the theorem admit non-primitive cohomogeneity one actions. Indeed,  $\s^2\times \HP^n$ and $\s^2\times (\Gtwo/\SO(4))$ admit obvious non-primitive cohomogeneity one actions via a cohomogeneity one action on one factor and a transitive action on the other. 

\begin{proposition}\label{pro:nonprimitiveaction}
The unique non-trivial $\HP^n$-bundle over $\s^2$ admits a non-primitive cohomogeneity one action.
\end{proposition}

\begin{proof}  
As shown by Iwata \cite{Iwata78}, the natural action of $\SU(n+1)$ on $\HP^n$ is cohomogeneity one.  We identify $\HP^n$ with the quotient of $\s^{4n+3}$ obtained by identifying ${v}\in \mathbb{H}^{n+1}$ with ${v} p$ for any unit quaternion $p$.  Consider the action of $L = \SU(n+1)\times \sone$ on $\HP^n$ given by $(A, e^{i\theta})\ast [v] = \left[e^{i\theta/2} A v\right]$.  Since this action extends a cohomogeneity one action, it has cohomogeneity at most one.  On the other hand, it is easy to see that this action cannot move the point $[1:i:0:...:0]$ to $[1:j:0:...:0]$, so is not transitive.

Consider the natural inclusion $L\subseteq \SU(n+1)\times \SU(2)$ and the cohomogeneity one $G$-manifold $M_L \times_{L} G = M_L \times_{\SU(n+1)\times \sone} (\SU(n+1)\times \SU(2)) \cong M_L\times_{\sone}  \SU(2)$.  This quotient of $\HP^n\times \s^3$ by $\sone$ appears in \cite[Proposition 4.4]{DeVito18} and is shown there to be the total space of the non-trivial $\HP^n$ bundle over $\s^2$.
\end{proof}

\section{Generalities on primitive actions}\label{sec:primitive}

In this section, we collect some results on primitive cohomogeneity one actions that will be used in the proofs of Theorems \ref{thm:QP} and \ref{thm:4periodic}. The proofs rest on a much wider classification due to Frank \cite{Frank13} in the case where the $G$-action is primitive.

\begin{theorem}[Frank]\label{thm:Frank} 
Suppose $M$ is a minimal and primitive cohomogeneity one manifold with group diagram $H \subseteq K^\pm \subseteq G$. If $M$ has positive Euler characteristic and the codimensions $k_\pm$ of the singular orbits satisfy $k_+ \equiv 0 \bmod{2}$ and $k_- \geq 3$, then at least one of the following applies.
\begin{enumerate}
\item\label{Frank1}  $M$ is diffeomorphic to a compact rank one symmetric space, a Grassmannian, $\Sp(n)/(\Sp(n-k+1)\Un(k))$, or $\SO(n+1)/(\SO(n-2k+1)\Un(k))$. Moreover, the action is linear, up to equivalence.

\item\label{Frank2}  $G = G_1\times \SU(2)$ some some simple group $G_1$.   

\item\label{Frank3} $G$ is an exceptional Lie group.

\item\label{Frank4} $M$ has a diagram with $G$ and $H$ listed in Table \ref{table:FrankAppendix}
\end{enumerate}
\end{theorem}

\begin{center}
\begin{table}[h]
\renewcommand*{\arraystretch}{1.5}
\begin{tabular}{|c|c|c|c|c|}
\hline
$G$ &  $H$ & $\dim M$ & $\chi(M)$ & $\frac{\dim(M)}{\chi(M)-1}$\\
\hline
\hline
$\SU(3)$ &  $\sone$ & $8$ & $3$ & $4$\\
$\SU(3)$ &  $\sone$ or  $\mathbb{Z}_3 \SO(2)$ & $8$ & $6$ & $\frac{8}{5}$\\
$\SU(4)$ &  $\sone \SU(2)$ & $12$ & $10$ & $\frac{4}{3}$\\
\hline
$\SU(n), n\geq 4$ & $\sone \SU(n-2)$ & $ 4n-4$  & $2n$ & $\frac{4n-4}{2n-1}$\\
\hline
$\SO(7)$ or $\Spin(7)$ & $\SU(3)$ & $14$ & $8$ & $\mathbf{2}$\\
$\Spin(7)$ & $\sone \SU(2)$ & $ 18$ & $14$ & $ \frac{18}{13}$\\
$\Spin(7)$ & $\SU(3)$ & $ {14}$ & ${2}$ & $\mathbf{14}$\\
$\SO(9)$ & $\sone \SU(3)$ & $28 $ & $16$ & $ \frac{28}{15} $\\
$\Spin(9)$ & $\sone \SU(2)^2$ & $ 30 $ & $48$ & $\frac{30}{47}$\\
\hline
$\SO(2n+1), n\geq 3$ & $\sone\SO(2n-3)$ & $8n-6$ & $2n(n+1)$ & $ \frac{8n-6}{2n^2 + 2n -1}$\\
$\SO(2n+1), n\geq 3$ & $\gT^2 \SU(n-2)$ & $n^2 +5n-4$ & $n2^n$ & $\frac{n^2 + 5n-4}{n2^n -1}$\\
\hline
$\Sp(2)$ & $\mathbb{Z}_2 \Sp(1)$ & $8$ & $8$ & $\frac{8}{7}$\\
\hline
$\Sp(n), n\geq 2$ & $\Sp(n-2)\Sp(1)$ & $8n-8$ & $n(2n-1)$ & $\frac{8n-8}{2n^2 - n - 1}$\\
$\Sp(n), n\geq 2$ & $\sone \Sp(n-2)$ & $8n-6$ & $2n^2$ & $\frac{8n-6}{2n^2 - 1}$\\
$\Sp(n), n\geq 2$ & $\sone \Sp(n-2)$ & $8n-6$ & $2n$ & $\frac{8n-6}{2n-1}$\\
$\Sp(n), n\geq 2$ & $\gT^2 \Sp(n-3)$ & $12n-6$ & $4n(n-1)$ & $\frac{12n-6}{4n^2 -4n -1}$\\
\hline
$\SO(8)$ & $\SU(4)$ & $14$ & $8$ & $\mathbf{2}$\\
$\SO(10)$ & $\sone \SU(4)$ & $30$ & $16$ & $\mathbf{2}$\\
\hline
$\SO(2n), n\geq 4$ & $\sone \SU(n-1)$ & $n^2 + n$ & $2^{n+1}$ & $ \frac{n^2 + n}{2^{n+1}-1}$\\
$\SO(2n), n\geq 4$ & $\gT^2 \SU(n-2)$ & $n^2 + 3n-4$ & $n2^{n-1}$ & $\frac{n^2 + 3n-1}{n2^{n-1} - 1}$\\
\hline
\end{tabular}\caption{Primitive cohomogeneity one manifolds with positive Euler characteristic}\label{table:FrankAppendix}
\end{table}
\end{center}

If Case 1 occurs in Theorem \ref{thm:Frank}, then Theorems \ref{thm:QP} and \ref{thm:4periodic} follow immediately from previous classifications (see Sections \ref{sec:QPcase1} and \ref{sec:4periodiccase1}). In the remaining cases of Frank's theorem, $G$ is simple or a product of a simple group and $\SU(2)$.  Since subgroups of compact Lie groups can have at most finitely many components, we  have the following.

\begin{lemma}\label{lem:pi1finite} 
Suppose $M$ is a simply connected manifold with $\chi(M) > 0$ and that $M$ does not arise in Case 1 of Theorem \ref{thm:Frank}.  For any primitive cohomogeneity one action by a compact Lie group $G$, the regular orbit $G/H$ has finite fundamental group.
\end{lemma}

The following further constrains what can occur in Case 2.

\begin{lemma}\label{lem:morecase2}
If $M$ is as in Case 2 of Frank's classification, then $k_- = 3$ and all of the following hold:
	\begin{enumerate}
	\item\label{lem:morecase2:K-} Up to cover, $K^- = K_1^-\times \Delta \SU(2)$ where $K_1^-\subseteq G_1$, $K^-_1$ has corank $1$ in $G_1$, and the $\Delta \SU(2)$ factor of $K^-$ is diagonally embedded in $G_1\times \SU(2)$. 
	\item The group $K^+$ has the form $K^+_1\times \sone$ with $K^+_1\subseteq G_1$ and $\sone\subseteq \SU(2)$.
	\item The group $H = K_1^-\times \Delta \sone$ for the natural $\Delta \sone\subseteq \Delta \SU(2)$.
	\item\label{lem:morecase2:bundle} In the bundle $\s^2\rightarrow G/H\rightarrow G/K^-$, the map $\pi_\ast(\s^2)\rightarrow \pi_\ast(G/H)$ is injective.
	\end{enumerate}
\end{lemma}

\begin{proof}  
The claim that $k_- = 3$ and the first and third statements follow immediately from Frank \cite[Section 2]{Frank13}.

For the second assertion, we begin by noting that since $K^+$ has full rank in $G = G_1\times \SU(2)$, it follows that $K^+ = K^+_1\times K^+_2$ where $K^+_1\subseteq G_1$ and $K^+_2\subseteq \SU(2)$ both have full rank.  Since $k_- > 2$, $K^+$ is connected, so $K^+_2 = \sone$ or $K^+_2 = \SU(2)$.  On the other hand, by first taking the quotient by the $K^-_1$ action, $\s^{k_+ - 1} = K^+/H = (K^+_1\times K^+_2)/(K^-_1\times \Delta\sone)$ fits into a bundle $K^+/H\rightarrow K^+_1/K^-_1$ with fiber $K^+_2/\sone$.  If $K^+_2 = \SU(2)$, this is a bundle of the form $\s^2\rightarrow \s^{k_+ - 1}\rightarrow K^+_1/K^-_1$.  Since $k_+-1$ is odd, there is no such bundle.  Thus, $K^+ = K^+_1\times \sone$, giving the first assertion.

For the fourth claim, since $K^- = K^-_1\times \Delta \SU(2)$, $G/K^-$ is canonically diffeomorphic to $G_1/K^-$.  The bundle $\s^2 = K^-/H\rightarrow G/H\rightarrow G_1/K^-_1$ has a section obtained by mapping $gK^-_1$ to $[g]\in G/H$, so the induced map $\pi_\ast(\s^2)\rightarrow \pi_\ast(G/H)$ is injective.
\end{proof}

The last two lemmas in this section will be applied in the proof of Theorem \ref{thm:QP}.

\section{Proof of Theorem \ref{thm:QP}}\label{sec:ProofQP}

In this section, we classify even-dimensional $\QP^n_k$ with cohomogeneity one structures. That is, we are given a simply connected, closed manifold $M$ of even dimension such that $H^*(M;\Q) \cong \Q[x]/(x^{n+1})$ where $x$ has even degree $k$. We are also given an almost effective cohomogeneity one action on $M$ by a connected group $G$ with group diagram $H \subseteq K^\pm \subseteq G$. Moreover, as explained in Section \ref{sec:preliminaries}, we may assume the $G$-action is minimal in the sense that no proper normal subgroup acts orbit equivalently. The aim is to prove that $M$ is equivariantly diffeomorphic to a compact rank one symmetric space (CROSS) equipped with a linear action. 

Part of this involves the classification when $n = 1$ (the even-dimensional rational sphere case). The other part requires showing that the parameters $n$ and $k$ are standard, i.e., that $k \in \{2,4\}$ or that $(k,n) = (8,2)$. The classification of Uchida and Iwata on cohomogeneity one rational cohomology CROSSes then implies the claimed equivariant diffeomorphism rigidity.

By what has been shown so far, after possibly swapping $K^+$ and $K^-$, we may assume all of the following:
	\begin{enumerate}
	\item $\rank(G) - 1 = \rank(H) \leq \rank(K^-) \leq \rank(K^+) = \rank(G)$ (see Section \ref{sec:preliminaries}).
	\item $G/K^+$ and $G/K^-$ are orientable (see Propositions \ref{pro:bothsingnonorientable} and \ref{pro:onesingnonorientable}).
	\item $k_+$ is even and $k_- \geq 3$ (see Proposition \ref{pro:s1s1}).
	\item The action is primitive (see Section \ref{sec:nonprimitiveQP}).
	\item If $n \geq 2$ and $k \geq 6$, then $k_\pm \geq 3$, both $G/K^\pm$ and $G/H$ are simply connected, and both $K^\pm$ and $H$ are connected (see Lemmas \ref{pro:orbittype} and \ref{lem:pi1finite} and Section \ref{sec:preliminaries}).
	\end{enumerate}

In particular, we may apply the classification of Frank (Theorem \ref{thm:Frank}, cf. Lemma \ref{lem:morecase2}). The rest of the proof of Theorem \ref{thm:QP} is carried out by stepping through the four cases in Frank's conclusion. Recall that our task is to classify the equivariant diffeomorphism type if $n = 1$ and to prove that $n \geq 2$ and $k > 4$ only if $(k,n) = (8,2)$.

\subsection{Proof of Theorem \ref{thm:QP}, Case 1}\label{sec:QPcase1}
The theorem follows unless $M$ is diffeomorphic to $\Sp(n) / (\Sp(n - m + 1)\Un(m))$ or $\SO(n+1)/(\SO(n-2m+1)\Un(m))$, and these latter two spaces do have not have singly generated rational cohomology (see, e.g., Kapovitch-Ziller \cite{KapovitchZiller04}).

\subsection{Proof of Theorem \ref{thm:QP}, Case 2}\label{sec:QPcase2}

In this subsection, we classify the case where $M$ is a $\QP^n_k$ as in Case 2 of Frank's theorem. We do this in two steps, according to whether $n = 1$ (rational sphere case) or $n \geq 2$.

\begin{lemma}
If Case 2 occurs in Theorem \ref{thm:Frank} and $M^k$ is an even-dimensional rational sphere, then one of the following happens:
\begin{enumerate}

\item $M$ is equivariantly diffeomorphic to $\s^4$  with the action of $G=\SU(2)\times \SU(2) = \Spin(4)$  via the suspension of the natural transitive $G$ action on $\s^3$

\item $M$ it is equivariantly diffeomorphic to $\s^{4m+2}\subseteq \mathbb{H}^m \oplus \operatorname{im}(\HH)$ equipped with the action of $G = \Sp(m)\times \Sp(1)$ given by $(A,p)\ast({v},{w}) = (A{v} p^{-1}, p{w} p^{-1})$.

\end{enumerate}
\end{lemma}

\begin{proof}
From the low dimensional classification of cohomogeneity one actions \cite{Hoelscher10}, we see that when $k=2$, there is no primitive cohomogeneity one action of a semisimple group $G$ on $\s^2$.  Likewise, when $k = 4$ and $G$ is a product of two simple groups, $G = \SU(2)\times \SU(2)$ acting as claimed.  Thus, we assume $k \geq 6$.

Recall $k_- = 3$.  We first claim that $G/K^+ = \s^2$.  To see this, notice first that we have $2 = \chi(M) = \chi(G/K^+)$.  Since $G/K^+$ is simply connected, this implies $G/K^+$ is rationally a sphere.  Now, consider the orientable bundle $\s^2\rightarrow G/H\rightarrow G/K^-$.  The rational Euler class vanishes automatically, so we conclude that $H^\ast(G/H;\Q) \cong H^\ast(\s^2\times G/K^-;\Q)$ as groups.  Now, considering the Mayer-Vietoris sequence in rational cohomology associated to the double disc bundle decomposition of $M$, we see that $H^2(G/K^+;\Q)\neq 0$, otherwise $H^3(M;\Q)\neq 0$, contradicting the fact that $M$ is a $\QP^1_k$ with $k > 3$.  Since $G/K^+$ is a rational sphere, it must be diffeomorphic to $\s^2$.  From the form $K^+ = K^+_1\times \sone\subseteq G_1\times \SU(2)$, we now see $K^+_1 = G_1$.  Also, note that $k_+ = k - 2$ because $\dim G/K^+ = 2$.

Since $k > 4$, the bundle $\s^{k-3}\rightarrow G/H\rightarrow G/K^+ = \s^2$ has trivial Euler class, so $G/H$ has the integral cohomology ring of $\s^2\times \s^{k-3}$.  It now follows from the bundle $\s^2\rightarrow G/H\rightarrow G/K^-$ that $G/K^- \cong G_1/K^-_1$ is rationally $\s^{k-3}$.  We note that the embedding of $K^-_1$ in $G_1$ extends to an embedding (up to cover) of $K^-_1\times \SU(2)$, which gives a bundle $\SU(2)\rightarrow G_1/K^-_1\rightarrow G_2/K^-_1\times \SU(2)$, from which it follows that $4|k-2$.

Further, the Euler class of the bundle $\s^2\rightarrow G/H\rightarrow G/K^-$ must be $2$-torsion.  So, $G/K^-$ is an integral cohomology sphere unless $H^3(G/K^-) \cong \Z_2$.  From Kapovitch-Ziller, there are no such rational spheres for which the dimension is congruent to $3$ mod $4$.  It follows that $G_1/K^-_1$ is diffeomorphic to $\s^{k-3}$, so $G_1/K^-_1\times \SU(2)$ is diffeomorphic to $\HP^{(k-6)/4}.$  This implies that $G_1 = \Sp(m)$ and $K^-_1 = \Sp(m-1)$ for $m = (k-2)/4$ (see \cite[Table 10, Theorem 2. pg. 265-266]{Onishchik97}).

Thus, writing $\SU(2) = \Sp(1)$, we have the following group diagram:  $\Sp(m-1)\times \Delta \sone\subseteq \Sp(m-1)\times \Sp(1), \Sp(m)\times \sone \subseteq \Sp(m)\times \Sp(1)$.  Since $k\cong 2\pmod{4}$, this group diagram is as claimed in the lemma.

\end{proof}

With the case $n = 1$, complete, we proceed to the case $n \geq 2$. In this case, we show no $\QP^n_k$ with $k\geq 6$ occurs.  We begin with a proposition which will be useful later.

\begin{proposition}
If $M$ is an even-dimensional $\QP_k^n$ as in Case 2 of Theorem \ref{thm:Frank}, then $n = 1$ or $k \in \{2,4\}$.
\end{proposition}

\begin{proof}
Since $k_- = 3$, we have that either Case 2.a or 2.b of Proposition \ref{pro:orbittype} applies.

If Case 2.b applies, then $k_+ + 3 = \frac{n+1}{2}k + 1$ and $k-1\in\{k_+ - 1, 3, \frac{n+1}{2}k-1\}$. It follows in each of the three cases for $k-1$ that $k \in \{2,4\}$ or that $n = 1$.

If instead Case 2.a applies,then $\pi_\ast^\Q(G/H)$ has a graded basis with elements of odd degrees $k_+ - 1$, $2k_- -3 = 3$, and $k_+ + 1 = \frac{k}{2}(n+1)-1$ and elements of even degree $k$ and $2$. Applying Lemma \ref{lem:morecase2} to the bundle $\s^2\rightarrow G/H\rightarrow G/K^-$, we see the map $\pi_\ast(\s^2)\rightarrow \pi_\ast(G/H)$ is injective, so $\pi_\ast^\Q(G/K^-)$ has a graded basis consisting of elements of degrees $k_+ - 1$, $\frac{k}{2}(n+1)-1$, and $k$. 

We assume $k \geq 6$ and derive a contradiction. The fact that $\pi_2^\Q(\s^{k_+ -1}\times \QP^n_k) =\pi_4^\Q(\s^{k_+ - 1}\times \QP^n_k) = 0$ implies that $K^-$ must be simple with no torus factors (see, for example, \cite[Proposition 3.3]{DeVito17}).  Further note that $\dim \pi_{odd}^\Q(G/K^-) = 2$. Such pairs $(G,K^-)$ are cataloged in \cite[Table 11, pg. 270]{Onishchik97}, and one can easily see that the only such pair for which $K^-$ acts transitively on an even dimensional sphere and for which $K^-$ has corank $1$ in $G$ is $(G,K^-) = (\F_4, \Spin(7))$.  Since $K^-/H$ is an even dimensional sphere, we see $H = \Spin(6) = \SU(4)$.  Since $H$ is the isotropy group of the action of $K^+$ on an odd sphere, $K^+ = \SU(5)$.  This is a contradiction since the Borel - de Siebenthal classification of maximal maximum rank subgroups of simple groups implies that there is no $\SU(5)\subseteq \F_4$, even up to cover.
\end{proof}

\subsection{Proof of Theorem \ref{thm:QP}, Case 3}\label{sec:QPcase3}

In this subsection, we show there is no primitive cohomogeneity one action of an exceptional Lie group $G$ on an even-dimensional $\QP^n_k$ with $k\geq 6$. The proof is in two parts, according to whether $n = 1$ or $n \geq 2$.

\begin{proposition}  
There is no primitive cohomogeneity one action of an exceptional Lie group $G$ on an even-dimensional $\QP^1_k$, i.e., on a rational sphere of even dimension.
\end{proposition}

\begin{proof}
By the remarks at the beginning of this section, we have $k_- \geq 3$ and that $G/K^+$ is simply connected.  Since $2 = \chi(M) = \chi(G/K^+) + \chi(G/K^-)$, it follows that $G/K^+$ is a rational sphere and $\chi(G/K^-) = 0$.  This implies that $G = \Gtwo$ and $K^+ \cong \SU(3)$.  Since $K^+/H$ is a sphere, $H \cong \SU(2)$.  However $\SU(2)$ is not the isotropy group of any transitive action on an even dimensional sphere.  Hence $K^-/H$ must be an odd dimensional sphere, which implies $\operatorname{rank} K^- = \rk \Gtwo$.  This contradicts the fact that $\chi(G/K^-) = 0$.
\end{proof}

\begin{proposition}
There is no primitive cohomogeneity one action of an exceptional Lie group $G$ on an even-dimensional $\QP^n_k$ with $n\geq 2$ and $k\geq 6$.
\end{proposition}

\begin{proof}
First suppose that $k_-$ is even. By the remarks at the beginning of Section \ref{sec:ProofQP}, we may assume that both $k_\pm\geq 3$ and hence that both $G/K^\pm$ are simply connected.  From Proposition \ref{pro:orbittype}, each $G/K^\pm$ is a  $\QP^{m_\pm}_k$ for some integers $m_\pm$ and $\pi_\ast^\Q(G/H)$ is three dimensional.  Since $k\geq 6$, it follows from Kapovitch-Ziller \cite{KapovitchZiller04} that $G/K^+ = \s^{m}$ (with $m$ even) or $G/K^+ = \OP^2$, and similarly for $G/K^-$.  Hence either $G = \Gtwo$ with $K^\pm \cong \SU(3)$ or $G=\F_4$ with $K^\pm \cong \Spin(9)$.  Since $K^+/H$ is an odd dimensional sphere, $H \cong \SU(2)$ in the first case and $H = \Spin(7)$ in the second case. We prove that both of these cases lead to a contradiction. Indeed, in the first, $G/H \cong \Gtwo/\SU(2)$ is diffeomorphic to the total space of the unit tangent bundle over $\s^6$. Hence $G/H$ is a rational $\s^{11}$, contradicting the fact that $\pi_\ast^\Q(G/H)$ has three generators. In the second case, $G/H\cong \F_4/\Spin(7)$, which is diffeomorphic to the total space of the unit tangent bundle over $\OP^2$. Hence $G/H$ is a rational $\s^8 \times \s^{23}$ with $H^{16}(G/H;\Q) = 0$.  The induced map $\Q = H^{16}(G/K^\pm;\Q)\rightarrow H^{16}(G/H;\Q)$ must have non-trivial kernel and hence $\dim H_{16}(M;\Q) \geq 2$ by the Mayer-Vietoris sequence associated to the double disc bundle decomposition. This is again a contradiction, so the proposition holds when $k_-$ is even.

Next assume that $k_-$ is odd and $k_- \geq 7$. By Proposition \ref{pro:orbittype}, $\pi_2^\Q(G/H) = \pi_4^\Q(G/H) = 0$. Hence $H$ is simple.  The only simple groups which are the isotropy groups of transitive actions on an even and odd dimensional sphere are $H = \Spin(6) = \SU(4)$ and $H = \SU(3)$.  Because $\rk H = \rk G - 1$ and there are no exceptional groups of rank $3$, we must be in the case where $G = \F_4$ and $H = \SU(4)$. In particular, $\dim \pi_{odd}^\Q(G/H) = 5$, so we must be in case 2a of Proposition \ref{pro:orbittype}.  Further, $\pi_9^\Q(G/H) = 0$ since $\pi_9^\Q(\F_4)=0$.   On the other hand, the fact that $H = \SU(4)$ implies that $K^+ = \SU(5)$ by Table \ref{transsphere}. Hence $k_+ = 10$, and so by Proposition \ref{pro:orbittype}, $\pi_9^\Q(G/H) \neq 0$, a contradiction.

Next assume that $k_- = 5$.  Then $\pi_4^\Q(G/H)\cong \Q$ and $\pi_2^\Q(G/H) = 0$, so $H$ is semisimple with precisely two simple factors.  Since $K^-/H\cong \s^4$, $H$ is, up to cover, $\SU(2)\times \SU(2)$.  Thus, $\rank G = \rank H + 1 = 3$, contradicting the fact that $G$ is exceptional.

Finally assume $k_- = 3$.  Since $k-1>3$, $\pi_3^\Q(G/H)\neq 0$ which implies $H$ is a torus.  Further, $\pi_2^\Q(G/H) \cong \Q$, so $H = \sone$ and $G = \Gtwo$.  But then $nk = \dim M = \dim G/H + 1 = 14$, so we have a contradiction to the assumptions that $n\geq 2$, $k\geq 6$, and $k \equiv 0 \bmod{2}$.
\end{proof}

\subsection{Proof of Theorem \ref{thm:QP}, Case 4}\label{sec:QPcase4}

In this subsection, we conclude the proof of Theorem \ref{thm:QP} by dealing with the last possibility in Theorem \ref{thm:Frank}.

\begin{proposition}\label{pro:s14}
If $M$ is an even-dimensional $\QP^n_k$ and arises in Case 4 of Theorem \ref{thm:Frank}, then either $k \in \{2,4\}$ or $M$ is given by the $G= \Spin(7)$ action on $\s^{14}\subseteq \mathbb{R}^8\oplus \mathbb{R}^7 \cong \mathbb{R}^{15}$ with $G$ acting as the sum of the spin representation and the standard representation.
\end{proposition}

We may assume $k \geq 6$. We examine Table \ref{table:FrankAppendix} and compute the quantity $\frac{\dim M}{\chi(M) - 1}$. For a $\QP^n_k$, this quantity is equal to $k$, so it must be integral and at least six. By inspection, this only occurs in the case where $G = \Spin(7)$, $H = \SU(3)$.  In this case, according to Frank \cite{Frank13}, we have $K^- = \Gtwo$ and $K^+ = \Spin(6)$.  This is the diagram corresponding to the $G= \Spin(7)$ action on $\s^{14}\subseteq \mathbb{R}^8\oplus \mathbb{R}^7 \cong \mathbb{R}^{15}$ with $G$ acting as the sum of the spin representation and the standard representation, so is a linear action on a standard sphere.

\section{Proof of Theorem \ref{thm:4periodic}}\label{sec:Proof4periodic}

In this section, we show that there are no primitive cohomogeneity one actions on any simply connected closed rational $\s^2\times \HP^n$.  This will complete the classification of even-dimensional, simply connected, closed cohomogeneity one manifolds $M$ with four-periodic rational cohomology and positive Euler characteristic.

As in the proof of Theorem \ref{thm:QP}, we assume the cohomogeneity one action is minimal, almost effective, and by a connected group $G$. Let $H \subseteq K^\pm \subseteq G$ be a group diagram for this action. Since again we are dealing with positive Euler characteristic, we again may assume all of the following:

	\begin{enumerate}
	\item $\rank(G) - 1 = \rank(H) \leq \rank(K^-) \leq \rank(K^+) = \rank(G)$ (see Section \ref{sec:preliminaries}).
	\item $G/K^+$ and $G/K^-$ are orientable (see Propositions \ref{pro:bothsingnonorientable} and \ref{pro:onesingnonorientable}).
	\item $k_+$ is even and $k_- \geq 3$ (see Proposition \ref{pro:s1s1}).
	\item The action is primitive (see Section \ref{sec:nonprimitiveQP}).
	\end{enumerate}

We again apply the classification of Frank, and we again conclude the proof by stepping through the four cases in the conclusion of Frank's theorem.

\subsection{Proof of Theorem \ref{thm:4periodic}, Case 1}\label{sec:4periodiccase1}

None of the homogeneous spaces in Case 1 have the same rational cohomology as $\s^2 \times \HP^n$ (see, for example, \cite{DeVito18}), so this case cannot occur.

\subsection{Proof of Theorem \ref{thm:4periodic}, Case 2}\label{sec:4periodiccase2}

In this subsection, we prove that there is no primitive cohomogeneity one rational $\s^2\times \HP^n$ in Case 2 of Theorem \ref{thm:Frank}. We start with some preliminary observations.

By Lemma \ref{lem:morecase2}, $k_- = 3$, so we have that $k_-$ is odd and $k_+$ is even. By Theorem \ref{GHtype}, we know that $\pi_*^\Q(\mathcal{F})$ has dimension five and is generated by elements in even degrees $2$ and $2(k_+ + 1)$ and odd degrees $3$, $k_+ - 1$, and $k_+ + 1$. Recall that $\pi_*^\Q(M)$ has dimension four with generators in degrees $2$, $3$, $4$, and $4n + 3$. Also by Lemma \ref{lem:morecase2}, $G/K^-$ is diffeomorphic to a homogeneous space $G'/K'$ where $G'$ is simple, $K'$ has corank $1$ in $G'$, and the $K'$ action on $G'$ extends to $K'\times \SU(2)$ homogeneous action.

By Proposition \ref{connecthom} and the fact that $k_+ \geq 2$, the connecting homomorphism $\partial:\pi_3^\Q(M)\rightarrow \pi_2^\Q(\mathcal{F})$ is zero and the connecting homomorphism $\partial:\pi(M)_{2(k_+ + 1) + 1}\Q\rightarrow \pi_{2(k_+ + 1)}^\Q(\mathcal{F})$ is non-trivial. The latter statement implies $2(k_+ + 1) + 1 = 4n + 3$ and hence $k_+ = 2n$. 

Next, the connecting homomorphism $\partial:\pi_2^\Q(M)\rightarrow \pi_1^\Q(\mathcal{F})$ is automatically trivial if $k_+  > 2$ and is non-trivial if $k_+ = 2$. Indeed, if it is trivial in the latter case, then $\pi_1^\Q(G/K^-) \cong \pi_1(G/H) \cong \Q$, contradicting the fact that $G/K^-$ is diffeomorphic to a homogeneous space of a simple group.

Finally, the connecting homomorphism $\pi_4^\Q(M)\rightarrow \pi_3^\Q(\mathcal{F})$ is non-trivial. Indeed, if it is trivial, then $\pi_4^\Q(G/H)$ is non-trivial and $\dim \pi_3^\Q(G/H)\geq 2$.  From the bundle $\s^2\rightarrow G/H\rightarrow G/K^-$, we see that $\pi_3^\Q(G/K^-)$ and $\pi_4^\Q(G/K^-)\cong \Q$ are both non-trivial. On the other hand, $G/K^-$ is diffeomorphic to the homogeneous space $G_1/K^-_1$ where $G_1$ is simple. In particular, the map $\pi_3^\Q(K^-_1) \to \pi_3^\Q(G_1)$ is surjective (see \cite[Section 10, p. 58, and Theorem 2, p. 257]{Onishchik97}), which implies that $\pi_3^\Q(G_1/K^-_1)$ and $\pi_4^\Q(G_1/K^-_1)$ cannot both be non-trivial. This is a contradiction.

With the connecting homomorphism computed, we can compute $\pi_*^\Q(G/H)$ using the long exact homotopy sequence for the fibration $\mathcal F \to G/H \to M$. This naturally splits into two cases, so we conclude the proof by proving the following two lemmas.

\begin{lemma}
There is no primitive cohomogeneity one rational $\s^2\times \HP^n$ in Case 2 of Theorem \ref{thm:Frank} such that $k_+ = 2$.
\end{lemma}

\begin{proof} 
By the comments above, we have that $\pi_*^\Q(G/H)$ has dimension three with generators in degrees $2$, $3$, and $3$. By inspection of the Sullivan minimal model, it is straightforward to see that $G/H \simeq_\Q \s^2 \times \s^3$. Using \cite[Theorem 3.1, Case 1]{DeVito17}, it follows from \ref{lem:morecase2} that up to cover, $G = \SU(2)\times \SU(2)$, $K^-_0 = \Delta \SU(2)$, and $H_0 = \Delta \sone$. Finally, from the bundle $\s^{1}\rightarrow G/H\rightarrow G/K^+$, we see that up to cover, $G/K^+$ is a rational $\s^2\times \s^2$.  Since $K^+$ is connected, we must have $K^+ = \sone\times \sone\subseteq \SU(2)\times \SU(2)$. 

This setup is studied in Hoelscher \cite[Example $Q^6_A$]{Hoelscher10}.  In particular, there are precisely two simply connected cohomogeneity  one manifolds of this form, having group diagrams $H = \mathbb{Z}/l\mathbb{Z} \cdot \Delta \mathrm{S}^1$, $K^+ = \sone\times \sone$, $K^- =  \mathbb{Z}/l\mathbb{Z} \cdot \Delta \SU(2)$, and $G =  \SU(2)\times \SU(2)$  with $l \in \{1,2\}$.  According to Uchida \cite{Uchida77}, the $l=1$ case corresponds to the natural $\SO(4)$ action on the Grassmannian $\SO(7)/\SO(2)\times\SO(5)$, a rational $\mathbb{C}P^3$.  Likewise the case $l=2$ corresponds to the natural action of $\SO(4)$ on $\CP^3$.  In both cases, we have a contradiction to the fact that $M \simeq_\Q \s^2 \times \HP^1$, so the case $k_+ = 2$ does not occur.
\end{proof}

\begin{lemma}
There is no primitive cohomogeneity one rational $\s^2\times \HP^n$ in Case 2 of Theorem \ref{thm:Frank} such that $k_+ \geq 4$.
\end{lemma}

\begin{proof}
Note in this case that both $K^\pm$ and $H$ are connected by Section \ref{sec:preliminaries}. As in the previous case, we can compute $\pi_*^\Q(G/H)$. It has dimension five and is generated by elements in degrees $2$, $2$, $3$, $k_+ - 1$, and $k_+ + 1$. Applying Proposition \ref{lem:morecase2} to the fibration $\s^2 \to G/H \to G/K^-$, we see that $\pi_*^\Q(G/K^-)$ has dimension three with generators in degrees $2$, $k_+ - 1$, and $k_+ + 1$.

Now we consider the fiber bundle $\s^{k_+-1}\rightarrow G/H\rightarrow G/K^+$. We first claim that $\pi_{k_+}(G/K^+) = 0$. Indeed, the long exact homotopy sequence implies that highest odd degree for which $\pi_i^\Q(G/K^+)\neq 0$ is $k_+ + 1$. Since $k_+ \geq 4$ and since $G/K^+$ has bounded cohomology, we see by inspection of the Sullivan minimal model for $G/K^+$ that $\pi_{k_+}^\Q(G/K^+) \neq 0$ would imply that no differential can kill powers of the generator in degree $k_+$, a contradiction. Hence $\pi_{k_+}(G/K^+) = 0$ and the long exact homotopy sequence implies that $\pi_*^\Q(G/K^+)$ has dimension four and is generated in degrees $2$, $2$, $3$, and $k_++ 1$.

We now compute $G = G_1 \times \SU(2)$. On on hand, the fact that $G/K^- = G_1/K^-_1$ has consecutive non-trivial odd rational homotopy groups implies $G_1$ does as well. Since $G_1$ is simple, we have $G_1\in\{\SU(m), \SO(4m+2), \E_6\}$. On the other hand, Lemma \ref{lem:morecase2} implies that $G/K^+$ is diffeomorphic to $\s^2 \times G_1/K^+_1$. Hence $G_1/K^+_1$ is a rational $\mathbb{C}P^{\frac{k_+}{2} + 1}$, which means that $G_1\in \{\SO(2m+1), \SU(m), \Sp(m), \Gtwo\}$ for some  $m$ (see \cite[Table 10, pg. 265]{Onishchik97}). Putting these together, we have $G_1 = \SU(m)$ with $m  = \frac{k_+}{2} + 2 = n + 2$. 

Next it follows that $K^+_1 = \Un(m-1)$ since $G_1/K_1^+$ is a rational $\CP^m$. In addition, $K^-_1 = \Un(m-2)$ by Theorem \ref{thm:transsphere} since $\s^{k_+ - 1} = K^+/H = (K^+_1 \times \sone)/(K^-_1 \times \Delta\sone)$  and since $K^-_1$ acts almost effectively on $K^+/H$ by \cite[Lemma 1.3]{Frank13}. It follows that up to conjugacy, this is the form of \cite[Example 3.3, pg 157]{Uchida77}.  In particular, $M$ is rationally $\CP^{2n+1}$, contradicting the fact that $M\simeq_\Q \s^2\times \HP^n$.
\end{proof}

\subsection{Proof of Theorem \ref{thm:4periodic}, Case 3}\label{sec:4periodiccase3}
In this subsection, we prove that no simply connected, closed manifold $M$ with $M \simeq_\Q \s^2\times\HP^m$ admits a primitive cohomogeneity one action by an exceptional Lie 
group $G$.

Since $\pi_3^\Q(M) \cong \pi_{4n+3}^\Q(M)\cong \Q$,  Proposition \ref{connecthom} implies that $\pi_2^\Q(-)$ or $\pi_{4n+2}(-)$ of the loop space factor must be non-trivial.  If it is $\pi_2^\Q(-)$, then $\mathcal{F}\simeq_\Q \s^1\times \s^1\times \Omega \s^3$.  Since we are assuming $k_-  > 2$, this cannot occur.  It now follows that $\pi_3^\Q(G/H)$ is non-trivial.

Since $G$ is simple, this implies $H_0$ is a torus.  Since $k_+$ is even, it follows that $\pi_2^\Q(\mathcal{F})$ has dimension at most $1$, so $\rk H = \dim\pi_2^\Q(G/H_0) \leq 2$.  Since there are no exceptional Lie groups of rank $3$, and $H$ has corank $1$ in $G$, this implies $G = \Gtwo$ and $H_0 = \sone$.  Thus, $\dim G/H = 13$, so $M\simeq_\Q \s^2\times \HP^3$.  In particular, $\pi_{15}^\Q(M)\neq 0$.  By Proposition \ref{connecthom}, we must have $k_+ + k_- - 1 = 15$ or $k_+ + k_- -1 = 8$.  Now, $\pi_{11}^\Q(\Gtwo)\cong \pi_{11}^\Q(G/H)\cong \Q$, so it follows from the long exact sequence associated to $\mathcal{F}\rightarrow G/H\rightarrow M$ that $\pi_{11}^\Q(\mathcal{F})\neq 0$.  Because $k_+ + k_- -1\in \{8,15\}$, it follows that up to reordering $k_\pm$, that $k_+ = 7$ or $k_+ = 12$.  But $H_0 = \sone$ is not the isotropy group of any transitive action on a sphere of dimension $6$ or $11$, so we have a contradiction.

\subsection{Proof of Theorem \ref{thm:4periodic}, Case 4}\label{sec:4periodiccase4}

In this subsection, we conclude the proof of Theorem \ref{thm:4periodic} by classifying simply connected rational $\s^2 \times \HP^n$ that admit primitive cohomogeneity one actions by simple classical groups. Frank's classification implies that $M$ is one of the $G$-manifolds listed in Table \ref{table:FrankAppendix}. We use the fact that $\s^2 \times \HP^n$ has the same dimension and Euler characteristic as $\CP^{2n+1}$. Hence the quantity $\frac{\dim M}{\chi(M) - 1} = 2$. By inspection of the table, this only occurs when $$(G,H)\in \{(\SO(7), \SU(3)),(\Spin(7), \SU(3)), (\SO(8), \SU(4)), (\SO(10), \sone\times \SU(4))\}.$$  The first, second, and last cases appear in \cite{Uchida77}: they are actions on $\SO(11)/\SO(2) \times \SO(9)$, $\mathbb{C}P^7$, and $\mathbb{C}P^{15}$ respectively. All of these spaces are rational complex projective spaces and hence not rational $\s^2 \times \HP^n$.

In the remaining case, $(G,H) = (\Spin(8), \SU(4))$, we have $K^+ = \Un(4)$ and $K^- = \Spin(7)$.  It follows that $k_+ = 2$, $k_- = 7$, and so $\pi_2^\Q(\mathcal{F}) = 0$.  Since $\pi_3^\Q(G/H) = 0$ as well, it follows that $\pi_3^\Q(M) = 0$ as well.  This contradicts the fact that $M\simeq_\Q \s^2\times \HP^n$.


\end{document}